\documentclass[11pt]{article}
\usepackage{amsthm,amsmath,amssymb}
\usepackage[T1]{fontenc}
\usepackage[english]{babel}
\usepackage[square,numbers]{natbib}
\usepackage[colorlinks=true,linkcolor=red,urlcolor=blue]{hyperref}
\usepackage{tikz}
\usetikzlibrary{calc}
\newtheorem{theorem}{Theorem}[section]
\newtheorem{lemma}{Lemma}[section]
\newtheorem{corollary}{Corollary}[section]

\newtheorem{conjecture}{Conjecture}[section]
\newtheorem{example}{Example}[section]

\DeclareMathOperator{\inte}{int}
\DeclareMathOperator{\exte}{ext}
\newcommand{\symdif}{\bigtriangleup}
\begin{document}
\title{Almost generalized uniform matroids and excluded minors}
\author{Hyungju Park\thanks{Department of Mathematics, Seoul National University, Seoul 08826, Korea. E-mail address: \texttt{parkhyoungju@snu.ac.kr}}}
\date{August 30, 2023}
\maketitle
\begin{abstract}

We establish that matroids characterized by the Tutte polynomial $\sum_{i,j\ge 0}t_{i,j}x^iy^j$ with coefficients $t_{i,j}$ vanishing for $(i,j)\ge (k,l)$ precisely coincide with $(k,l)$-uniform matroids. This characterization implies that almost $(k,l)$-uniform matroids are exactly matroids with $t_{k,l}\le 1$ and $t_{i,j}=0$ if $(i,j)>(k,l)$. We also characterize excluded minors of almost $(k,l)$-uniform matroids in terms of Tutte polynomial coefficients. Finally, we construct an infinite family of excluded minors of almost $(k,l)$-uniform matroids which extend previously known cases of almost uniform and almost paving matroids.

\end{abstract}

\section{Introduction}
\emph{Paving matroids} are among the simplest examples of matroids. Due to the simplicity of paving matroids, several unsolved problems of matroids have been solved for this specific subclass. For example, Rota's basis conjecture \cite{MR2272246}, Stanley's conjecture on the $h$-vectors of matroids \cite{MR2950481} and the Merino-Welsh conjecture \cite{Ferroni2022a} are shown to hold for paving matroids. Some conjectures are disproved by a counterexample constructed in the class of paving matroids: a paving matroid whose Ehrhart polynomial has a negative coefficient is found by \citet{MR4396506}, which was the first counterexample to the Ehrhart positivity conjecture for matroid polytope \cite[Conjecture~2]{MR2556462} or more generally for integral generalized permutohedra \cite[Conjecture~1.2]{MR3869454}. Despite its simplicity, there is a conjecture suggesting that almost all matroids are paving matroids \cite[Conjecture~1.6]{MR2821559}. This provides a compelling reason for considering the class of paving matroids as a testing ground for conjectures in matroid theory.

A remarkable feature of the class of paving matroids is the elegant simplicity with which they can be characterized from various perspectives. We present several perspectives that have paved the way for the generalization of paving matroids. Paving matroids are most commonly defined by matroids with the girth(the size of its smallest circuit) equal to the rank which is translated to a simple restriction on matroid polytopes. \citet{MR3663497} introduced split matroids as a generalization of paving matroids by weakening the restriction of matroid polytopes of paving matroids and obtained results on the dimension and rays of Dressians.

Paving matroids can also be characterized in terms of their hyperplanes. A collection of subsets of a finite set $E$ is the set of hyperplanes of a paving matroid of rank $d$ if and only if it is a $(d-1)$-partition which was introduced by \citet{MR0099931}, see \cite[Proposition~2.1.24]{MR2849819}. \citet{MR4499860} provided another generalization of paving matroids in this direction. They found a subclass of split matroids, called elementary split matroids,  characterized by the hypergraph representation which generalizes the hyperplane characterization of paving matroids. In \cite{MR4499860}, authors provided another proof of five excluded minors of split matroids, which were found in \cite{MR4396506} and first proved by \citet{MR4224380}, and classified binary split matroids.

An alternative approach to defining paving matroids is through the excluded minor $U_{2,2}\oplus U_{0,1}$. The first generalization of paving matroids in terms of the excluded minor is $k$-paving matroids studied by \citet{MR1639777}. \citet{MR4290142} further extended this generalization with the notion of $(k,l)$-uniform matroids. A $(k,l)$-uniform matroid is defined by a matroid with no $U_{k,k}\oplus U_{0,l}$ minor. Ordinary paving matroids are $(2,1)$-uniform matroids and $k$-paving matroids are $(k+1,1)$-uniform matroids by \cite[proposition~2]{MR1639777}. In \cite{MR4290142}, it is shown that there are only finitely many $(k,l)$-uniform matroids that are $\mathbb{F}_q$-representable for each $(k,l)$ and binary $(2,2)$-uniform matroids are classified as a natural next step of the result of binary paving matroids \cite{MR943728} and binary $2$-paving matroids \cite{MR1639777}.

Finally, paving matroids are characterized by the support of their Tutte polynomials. \citet[Proposition~7.9]{MR0309764} proved that a matroid with the Tutte polynomial $\sum_{i,j}t_{i,j}x^iy^j$ is paving if and only if $t_{i,j}=0$ whenever $(i,j)\geq(2,1)$. Throughout this paper, the notation $(i,j)>(k,l)$ or $(i,j)\geq(k,l)$ will be used to denote that the inequality holds for each component. It is natural to ask which class of matroids satisfy $t_{i,j}=0$ for all $(i,j)\ge (k,l)$ for some fixed pair of positive integers $k,l$. Our first main result, proved in Section \ref{sec:2}, of this paper is that the answer is exactly the class of $(k,l)$-uniform matroids.

\begin{theorem}\label{uni}

For a pair of positive integers $(k,l)$, a matroid $M$ is $(k,l)$-uniform if and only if the Tutte polynomial $T_M=\sum_{i,j}t_{i,j}x^iy^j$ of $M$ satisfies $t_{i,j}=0$ for all $(i,j)\ge (k,l)$.

\end{theorem}

This new characterization of $(k,l)$-uniform matroids provides a way to study almost $(k,l)$-uniform matroids and their excluded minors. For a class of matroids $\mathcal{M}$, an almost $\mathcal{M}$ matroid is a matroid $M$ such that for every element $v$ of $M$, either $M-v$ or $M/v$ is contained in $\mathcal{M}$. Almost $\mathcal{M}$ matroids were studied for various class of matroids including regular \cite{MR1168966}, binary \cite{MR3131892}, graphic \cite{MR1900004} and paving matroids \cite{vega2020excluded}. In Section \ref{sec:3}, we also provide the Tutte polynomial coefficient characterization of almost $(k,l)$-uniform matroids.

\begin{theorem}\label{almuni}

A matroid $M$ is almost $(k,l)$-uniform if and only if the Tutte polynomial $T_M=\sum_{i,j}t_{i,j}x^iy^j$ of $M$ satisfies $t_{k,l}\le 1$ and $t_{i,j}=0$ if $(i,j)> (k,l)$.

\end{theorem}

In general, if $\mathcal{M}$ is a minor-closed class of matroids, the class of almost $\mathcal{M}$ matroids is also minor closed as observed in \cite{lara2020exploring, vega2020excluded}. Therefore, the class of almost $(k,l)$-uniform matroids is closed under taking minor. The result of the Theorem \ref{almuni} also manifests minor-closedness of $(k,l)$-uniform matroids from the deletion-contraction recursion of the Tutte polynomial. Next corollary states that excluded minors of almost $(k,l)$-uniform matroids also satisfy certain conditions on Tutte polynomial coefficients.

\begin{corollary}\label{exc}

A matroid $M$ is an excluded minor of almost $(k,l)$-uniform matroids if and only if $M\cong U_{k+1,k+1}\oplus U_{0,l}$, $M\cong U_{k,k}\oplus U_{0,l+1}$ or the following three conditions hold : 

\begin{itemize}
\item $M$ contains an element that is not a loop or a coloop.
    \item Its Tutte coefficients satisfy $t_{k,l}=2$ and $t_{i,j}=0$ if $(i,j)> (k,l)$.
    \item For every element $v\in M$ which is not a loop or a coloop, the $(k,l)$th coefficient $t_{k,l}^{M-v}$ of the Tutte polynomials of $M-v$ is $1$.
\end{itemize}

\end{corollary}

It follows from Corollary \ref{exc} that if $N$ is an excluded minor of almost $(k,l)$-uniform matroids, then $N\oplus U_{a,a}\oplus U_{0,b}$ is an excluded minor of almost $(k+a, l+b)$-uniform matroids and vice versa. Therefore it is enough to consider loopless and coloopless excluded minors of almost $(k,l)$-uniform matroids. In Section \ref{sec:cons}, we construct infinitely many loopless and coloopless excluded minors of almost $(k,l)$-uniform matroids extending previously known cases of uniform and paving matroids \cite[Theorem~3.13, Theorem~4.1]{vega2020excluded}. Let $\tau(M)$ denote the \emph{truncation} of matroid $M$, defined as the matroid where independent sets consist of those from $M$ that are not bases.

\begin{theorem}\label{exc1}

Let $N$ be a loopless and coloopless matroid of rank $k+m$ on the set of cardinality $k+l+m$ with $k,l>0$. Then the matroid $$\tau^m(N\oplus U_{k+m,k+l+m})$$ is a loopless and coloopless excluded minor of almost $(k,l)$-uniform matroids.
    
\end{theorem}

\section{Tutte polynomial coefficient characterization of $(k,l)$-uniform matroids\label{sec:2}}

For every unexplained term and notation related to matroids, we refer to Oxley's book \cite{MR2849819}. Let $M$ be a matroid on the set $\{1,\ldots,n\}$ which will be denoted by $[n]$ throughout the paper. For a basis $B$ of $M$ and an element $v\in B$, there is a unique cocircuit contained in $([n]-B)\cup\{v\}$. This cocircuit is called the \emph{fundamental cocircuit} of $v$ with respect to $B$ and will be denoted by $C^*(v,B)$. Dually, for $v\in [n]-B$, the unique circuit contained in $B\cup\{v\}$ is called the \emph{fundamental circuit} of $v$ with respect to $B$ and denoted by $C(v,B)$.

Given an order of the underlying set of a matroid $M$, an element $v\in B$ is \emph{internally active} with respect to $B$ if $v$ is the smallest element in the fundamental cocircuit $C^*(v,B)$. Dually, an element $v\in [n]-B$ is \emph{externally active} with respect to $B$ if $v$ is the smallest element in the fundamental circuit $C(v,B)$. The number of internally active elements of a basis $B$ is called \emph{internal activity} and will be denoted by $\inte(B)$. Similarly, \emph{external activity} of $B$ is denoted by $\exte(B)$. In this paper, we will simply say a basis $B$ is a \emph{$(i,j)$-basis} if $(\inte(B),\exte(B))=(i,j)$

The \emph{Tutte polynomial} $T_M(x,y)$ is a bivariate polynomial associated to a matroid defined by $\sum_{B}x^{\inte(B)}y^{\exte(B)}$ where $B$ runs over all bases of $M$. Although internal activity and external activity of a given basis depend on the order of the underlying set, the sum of all terms $x^{\inte(B)}y^{\exte(B)}$ is independent of the order of the underlying set. Hence, the Tutte polynomial is a polynomial invariant for matroids, where the coefficient $t^M_{i,j}$ of the monomial $x^iy^j$ corresponds to the count of bases of $M$ with internal activity $i$ and external activity $j$. It follows from the definition that the Tutte polynomial of a matoid $M$ satisfies the deletion-contraction recursion: $$
T_M(x,y) = \begin{cases}
    T_{M-v}(x,y) + T_{M/v}(x,y) & \text{if } v \text{ is not a coloop or a loop of } M \\
    xT_{M-v}(x,y) & \text{if } v \text{ is a coloop of } M \\
    yT_{M-v}(x,y) & \text{if } v \text{ is a loop of } M
\end{cases}.
$$ Therefore, if a matroid consists only of $j$ loops and $i$ coloops, then its Tutte polynomial is the monomial $x^iy^j$. We will frequently omit the superscipt $M$ from the notation $t^M_{i,j}$ if the specific matroid under reference is evident.

For a matroid $M$, each of the following conditions is equivalent to $M$ being a uniform matroid:

\begin{enumerate}
    \item     $M$ has no minor isomorphic to $U_{1,1}\oplus U_{0,1}$.
    \item Every hyperplane has nullity $0$, i.e., every flat of corank $1$ has nullity less than $1$.
    \item  If $T_M(x,y)=\sum t_{i,j}x^iy^j$ is the Tutte polynomial of $M$, then $t_{i,j}=0$ for all $(i,j)\ge (1,1)$.
\end{enumerate}
A slightly weaker form of each condition serves to define paving matroids:

\begin{enumerate}
    \item  $M$  has no minor isomorphic to $U_{2,2}\oplus U_{0,1}$.
    \item Every flat that is not a hyperplane is independent, i.e., every flat of corank $2$ has nullity less than $1$.
    \item  If $T_M(x,y)=\sum t_{i,j}x^iy^j$ is the Tutte polynomial of $M$, then $t_{i,j}=0$ for all $(i,j)\ge (2,1)$.
\end{enumerate}
Therefore, it is natural to ask if the following $(k,l)$-versions of each condition are equivalent:

\begin{enumerate}
    \item\label{item:1} $M$ has no minor isomorphic to $U_{k,k}\oplus U_{0,l}$.
    \item\label{item:2} Every corank $k$ flat has nullity less than $l$.
    \item\label{item:3}  If $T_M(x,y)=\sum t_{i,j}x^iy^j$ is the Tutte polynomial of $M$, then $t_{i,j}=0$ for all $(i,j)\ge (k,l)$.
\end{enumerate}

Recently, \citet{MR4290142} considered the class of \emph{$(k,l)$-uniform matroids} which is exactly defined by the statement \ref{item:1}: matroids with no $U_{k,k}\oplus U_{0,l}$ minor. The class of $(k,l)$-uniform matroids will be denoted by $\mathcal{U}(k,l)$. In \cite[Proposition~1.4]{MR4290142}, the equivalence of the statement \ref{item:1} and the statement \ref{item:2} is shown. Theorem \ref{uni} asserts that the statement \ref{item:3} on Tutte polynomial coefficients is also a necessary and sufficient condition for a matroid to be $(k, l)$-uniform.

The sufficiency of the condition \ref{item:3} is already proved by \citet[Proposition~6.10]{MR0309764} using induction. Another way to prove sufficiency is that if a matroid $M$ contains a minor isomorphic to $U_{k,k}\oplus U_{0,l}$, then the Tutte polynomial of $M$ contains a monomial $x^iy^j$ with $(i,j)\ge (k,l)$ by referring to the deletion-contraction recursion. Hence, it remains to prove the necessity. We will make use of the following statement \cite[Proposition~2.1.11]{MR2849819}.
\begin{lemma}\label{ccint}

If $C$ is a circuit and $C^*$ is a cocircuit of a matroid $M$, then $|C\cap C^*|\neq 1$.
\end{lemma}

\begin{proof}[Proof of Theorem \ref{uni}]

Let $M$ be a $(k,l)$-uniform matroid. Assume that there is a basis $B$ with $(\inte(B),\exte(B))\ge (k,l)$. Let $S$ be a subset of $B$ consisting of $k$ internally active elements of $B$. Then the closure of the set $B-S$ is a flat $F$ of corank $k$. By the second condition above,  $F$ has nullity less than $l$. This implies that $|F-(B-S)|<l$. Therefore, there is some element $v$ that is externally active with respect to $B$ and not contained in $F$. Then the intersection of the set $S$ and the fundamental circuit $C(v,B)$ is nonempty because $C(v,B)-\{v\}$ is not contained in $F$. Let $w$ be an element in the intersection $S\cap C(v,B)$.  Then the intersection of the fundamental cocircuit $C^*(w,B)$ and the fundamental circuit $C(v,B)$ is either $\{w\}$ or $\{v,w\}$. By Lemma \ref{ccint}, the intersection should be $\{v,w\}$. This leads to a contradiction since $v,w$ are the smallest elements in $C(v,B), C^*(w,B)$ respectively.

\end{proof}

By the maximality property of a matroid \cite[Theorem~1.3.1]{MR1989953}, there is a unique smallest basis with respect to the \emph{Gale order}, that is, there is a unique basis $\{a_1,<\cdots<a_d\}$ such that for every other basis $\{b_1<\cdots<b_d\}$, the inequality $a_i\le b_i$ holds for each $i$. It is clear from the definition that the smallest basis is a $(d,0)$-basis. Combining this fact with Theorem \ref{exc}, we get the following corollary.

\begin{corollary}\label{cor:max}

Every matroid $M$ on the set $[n]$ of rank $k$ is $(k,1)$-uniform and $(1,n-k)$-uniform. In particular, every matroid of rank $k$ or rank $k+m$ on the set $[k+l+m]$ is $(k,l)$-uniform.

\begin{proof}

If a matroid $M$ has rank $k$, then a basis $B$ has $k$ internally active element if and only if $B$ is the first basis of $M$. If $B'$ is a basis that is not the smallest, then there is an element $x\in B'-B$. By the basis exchange property of matroids, there is an element $y\in B-B'$ such that $B'\cup\{y\}\setminus\{x\}$ is a basis. Since $B\le B'$, it is clear that $y\le x$ and therefore $x$ is not internally active. Hence $M$ is $(k,1)$-uniform. $(1,n-k)$-uniformity is given by duality.
    
\end{proof}
    
\end{corollary}

\section{Almost $(k,l)$-uniform matroids and excluded minors}\label{sec:3}

Recall that for a given class of matroid $\mathcal{M}$, a matroid $M$ is \emph{almost $\mathcal{M}$-matroid} if for all element $v$ of the underlying set of $M$, either $M-v$ or $M/v$ is in the class $\mathcal{M}$. The class of all almost $\mathcal{M}$ matroids will be denoted by $A\mathcal{M}$. If a matroid consists only of loops and coloops, then following Lemma \ref{allem} settles Theorem \ref{almuni} for this specific case.

\begin{lemma}\label{allem}

The matroid $M=U_{i,i}\oplus U_{0,j}$ consisting only coloops and loops is not an almost $(k,l)$-uniform matroid if and only if $(i,j)>(k,l)$.

\begin{proof}

Deletion and contraction of an element $v$ yields the same minor if $v$ is a loop or a coloop. Therefore, it is enough to consider the deletion $M-v$  for every element $v$. By symmetry, it reduces to considering two cases of $v$ being a coloop and $v$ being a loop. If a coloop is deleted from $M$, the obtained matroid is $U_{i-1,i-1}\oplus U_{0,j}$. This matroid has a $U_{k,k}\oplus U_{0,l}$-minor if and only if $(i-1,j)\ge(k,l)$. And if a loop is deleted from $M$, the obtained matroid is $U_{i,i}\oplus U_{0,j-1}$ which has a $U_{k,k}\oplus U_{0,l}$-minor if and only if $(i,j-1)\ge (k,l)$. Combining these two statements, $M$ is not almost $(k,l)$ uniform if and only if either the inequality $(i-1,j)\ge(k,l)$ or the inequality $(i,j-1)\ge (k,l)$ holds. Therefore $M$ is not almost $(k,l)$-uniform if and only if $(i,j)>(k,l)$.

\end{proof}
    
\end{lemma}

The minor-closedness of almost $(k,l)$-uniform matroids is needed in the proof of Theorem \ref{exc1}.

\begin{lemma}\label{minorcl}

The class $A\mathcal{U}(k,l)$ of almost $(k,l)$-uniform matroids is minor-closed.

\begin{proof}

Since the class of $(k,l)$-uniform matroids are defined by an excluded minor, it is minor-closed. Therefore, the class of almost $(k,l)$-uniform matroids is also minor-closed by \cite[Theorem 3.2]{lara2020exploring}.
    
\end{proof}
    
\end{lemma}

Next lemma states that a $(i,j)$-basis becomes a $(i,j')$-basis for $j'\ge j$ after contracting an element that is not internally active with respect to the basis.

\begin{lemma}\label{lem:act}

Let $B$ be a $(i,j)$-basis of a matroid $M$ on the set $[n]$. If $v\in B$ is not an internally active element with respect to $B$, then the equality $\inte_M(B)=\inte_{M/\{v\}}(B-\{v\})$ and the inequality $\exte_M(B)\le\exte_{M/\{v\}}(B-\{v\})$ hold.

Dually, if $v\in [n]-B$ is not an externally active element with respect to $B$, then the inequality $\inte_M(B)\le\inte_{M-\{v\}}(B)$ and the equality $\exte_M(B)=\exte_{M-\{v\}}(B)$ hold.

\begin{proof}

Suppose that $v\in B$ is not internally active. It is a well-known fact \cite[Proposition 3.1.17]{MR2849819} that the collection of cocircuits of $M/v$ is exactly cocircuits of $M$ not containing $v$. For each $w\in B-\{v\}$, the fundamental cocircuit $C^*(w,B)$ does not contain $v$ and therefore remains the fundamental cocircuit $C^*(w,B-\{v\})$ in $M/v$. This implies that every element $w\in B-\{v\}$ is internally active with respect to $B-\{v\}$ in $M/v$ if and only if $w$ is internally active with respect to $B$ in $M$. This proves the equality $\inte_M(B)=\inte_{M/v}(B-\{v\})$.

Now, choose an element $w\in[n]-B$. Circuits of $M/v$ are minimal subsets among the collection $\{C\in \mathcal{C}(M)|C-\{v\}\}$ where $\mathcal{C}(M)$ is the collection of circuits of $M$, see \cite[Propostion 3.1.10]{MR2849819}. Thus, if the fundamental circuit $C(w,B)$ contains $v$, then $C(w,B)-\{v\}$ is the fundamental circuit of $w$ with respect to $B-\{v\}$ in $M/v$.

Assume that $v$ is not contained in $C(w,B)$. If there is a circuit $C'$ of $M$ such that $C'\setminus C(w,B)=\{v\}$, then both $C'$ and $C(w,B)$ are circuits of $M$ contained in $B\cup \{w\}$ which is a contradiction. Therefore, $C(w,B)$ is minimal among the collection $\{C\in \mathcal{C}(M)|C-\{v\}\}$ which implies that $C(w,B)$ is a circuit of $M/v$.

In both cases, an externally active element $w$ will remain externally active with respect to $B-v$ in the matroid $M/v$, thus proving the inequality $\exte_M(B)\le\exte_{M/v}(B-\{v\})$.

And the dual statement follows by considering the dual matroid $M^*$ and the dual basis $[n]-B$.
    
\end{proof}
    
\end{lemma}

The inequality $\exte_M(B)\le\exte_{M/\{v\}}(B-\{v\})$ can be strict. For example, if $M$ is the uniform matroid $U_{1,3}$ on the set $[3]$ and $B=\{2\}$, then $B$ is a $(0,1)$-basis since it has no internally active element and one externally active element $1$. Contracting $v=2$, the resulting matroid $M/v$ is $U_{0,2}$ and $B-\{v\}=\emptyset$ which is the unique basis of $M/v$. Then $B-\{v\}$ is a $(0,2)$-basis since every loop is an externally active element of every basis. Now, we are ready to prove Theorem \ref{almuni}.
 
\begin{proof}[Proof of  Theorem \ref{almuni}]

Suppose that $M$ satisfies $t_{k,l}=1$ and $t_{i,j}=0$ for $(i,j)>(k,l)$. For an element $v$ of $M$, the Tutte polynomial $T_M(x,y)$ of $M$ is equal to one of $xT_{M-v}$, $yT_{M-v}$ or $T_{M-v}+T_{M/v}$ by the deletion-contraction recursion. For each of these three cases, either $M-v$ or $M/v$ has vanishing Tutte polynomial coefficients $t_{i,j}$ for all $(i,j)\ge (k,l)$. This implies that $M$ is almost $(k,l)$-uniform by Theorem \ref{exc}.

To show the converse statement, first suppose that $t_{i,j}>0$ for some $(i,j)>(k,l)$. Then there is a basis $B$ satisfying $(\inte(B), \exte(B))>(k,l)$. This implies that $M$ is not $(\inte(B), \exte(B))$-uniform and has a minor isomorphic to $U_{\inte(B),\inte(B)}\oplus U_{0,\exte(B)}$ which is not almost $(k,l)$-uniform by Lemma \ref{allem}. Since almost $(k,l)$-uniform matroids are minor-closed by Lemma \ref{minorcl}, $M$ is not almost $(k,l)$-uniform. Now, suppose $t_{k,l}>1$. Then there are two distinct bases $B_1, B_2$ with internal activity $k$ and external activity $l$. Let $v$ be the maximal element of the symmetric difference $B_1\symdif B_2$ of two bases. Without loss of generality, we may assume $v\in B_1$.  Then $C(v,B_2)\cap (B_2-B_1)$ is nonempty because otherwise $C(v,B_2)$ is contained in $B_1$. This implies that $v$ is not externally active with respect to $B_2$ since $v$ is the maximal element of $B_1\symdif B_2$. Similarly, the intersection $C^*(v,B_1)\cap (B_2-B_1)$ is nonempty because otherwise the cocircuit $C^*(v,B_1)$ is contained in the cobasis $[n]-B_2$ which is a contradiction. Therefore, $v$ is not an internally active element of $B_1$. Hence $B_1-v$ remains a $(k,l')$-basis for $l'\ge l$ after contracting $v$ and $B_2$ remains a $(k',l)$-basis for $k'\ge k$ after deleting $v$ by Lemma \ref{lem:act}. Thus, both $M/v$ and $M-v$ are not $(k,l)$-uniform thereby showing $M$ is not almost $(k,l)$-uniform. Hence, it is necessary to satisfy $t_{k,l}=1$ and $t_{i,j}=0$ for $(i,j)>(k,l)$ if $M$ is almost $(k,l)$-uniform.
     
\end{proof}

An abundant source of examples of almost $(k,l)$-uniform matroids is the class of Schubert matroids. Schubert matroids, also known as nested matroids, shifted matroids, or generalized Catalan matroids, is a matroid of rank $d$ associated with a $d$-element subset $I=\{a_1<a_2,...,<a_d\}$ of the $[n]$. Bases of Schubert matroid associated to $I$ is all $d$-element subsets $\{b_1<b_2<\cdots<b_d\}$ that is larger than $I$ with respect to the Gale order, i.e., $a_i<b_i$ for all $i=1,...,d$.

\begin{figure}
\centering
\begin{tikzpicture}[scale=1]
  \draw[help lines] (0,0) grid (5,2);
  \draw[help lines] (2,2) grid (5,5);
  \draw[line width=2pt] (0,0) -- (1,0) -- (2,0) -- (3,0) -- (4,0) -- (5,0) -- (5,1) -- (5,2) -- (5,3) -- (5,4) -- (5,5);
  \draw[line width=2pt] (0,0) -- (0,1) -- (0,2) -- (1,2) -- (1,3) -- (2,3) -- (2,4) -- (2,5) -- (3,5) -- (4,5) -- (5,5);
  \draw[line width=2pt, color=red] (0,0) -- (1,0) -- (1,2) -- (1,2) -- (1,3) -- (2,3) -- (2,4) -- (2,5) -- (3,5) -- (4,5) -- (5,5);
  \node at (0.5,-0.2) {$y$};
  \node at (0.8,2.5) {$x$};
  \node at (1.8,3.5) {$x$};
  \node at (1.8,4.5) {$x$};
\end{tikzpicture}
\begin{tikzpicture}[scale=1]
  \draw[help lines] (0,0) grid (5,2);
  \draw[help lines] (2,2) grid (5,5);
  \draw[line width=2pt] (0,0) -- (1,0) -- (2,0) -- (3,0) -- (4,0) -- (5,0) -- (5,1) -- (5,2) -- (5,3) -- (5,4) -- (5,5);
  \draw[line width=2pt] (0,0) -- (0,1) -- (0,2) -- (1,2) -- (1,3) -- (2,3) -- (2,4) -- (2,5) -- (3,5) -- (4,5) -- (5,5);
  \draw[line width=2pt, color=blue] (0,0) -- (1,0) -- (2,0) -- (2,1) -- (2,2) -- (2,3) -- (2,4) -- (2,5) -- (3,5) -- (4,5) -- (5,5);
  \node at (0.5,-0.2) {$y$};
  \node at (1.5,-0.2) {$y$};
  \node at (1.8,3.5) {$x$};
  \node at (1.8,4.5) {$x$};
\end{tikzpicture}
\caption{The path corresponding to the $(3,1)$-basis and the $(2,2)$-basis}
\label{fig:sch}
\end{figure}
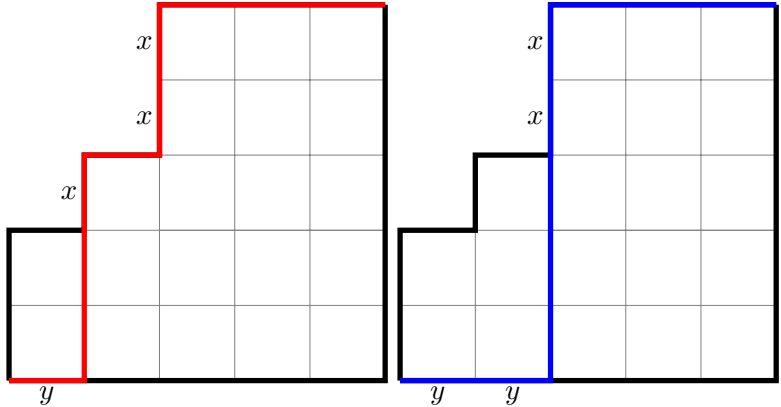

\begin{example}

The bases of a Schubert matroid are visualized by a lattice path bounded by two lattice paths. For example, suppose that $M$ is a Schubert matroid of rank $5$ on $[10]$ associated to $\{1,2,4,6,7\}$. Every lattice path from $(0,0)$ to $(5,5)$ with only increasing $x$ and $y$ coordinates, i.e., consisting only of east steps (E) and north steps (N), between the paths $P=E^5N^5$ and $Q=NNENENNEEE$, corresponds to a basis of the Schubert matroid $M$. For each lattice path the position of the north steps in the path determines the corresponding basis. In Figure \ref{fig:sch}, the red path corresponds to the basis $\{2,3,4,6,7\}$ and the blue path corresponds to the basis $\{3,4,5,6,7\}$. 

Bonin \cite[\S5]{MR2018421} discovered that the external activity of a lattice path, which corresponds to a basis, is the number of east steps shared with $P$ and the internal activity is the number of north steps shared with $Q$. The figure \ref{fig:sch} illustrates lattice paths corresponding to a $(3,1)$-basis and a $(2,2)$-basis. The figure clearly indicates that there is no $(i,j)$-basis with $(i,j)$ greater than $(3,1)$ or $(2,2)$ and no other $(3,1)$ or $(2,2)$-bases. Hence, $M$ is both almost $(3,1)$-uniform and almost $(2,2)$-uniform. In general, for a collection of $(k_1,l_1),...,(k_c,l_c)$ pairwise incomparable pair of positive integers, there is a Schubert matroid that is simultaneously almost $(k_1,l_1),...,(k_c,l_c)$-uniform.

\end{example}

A direct consequence of Theorem \ref{almuni} is Corollary \ref{exc} characterizing excluded minors of almost $(k,l)$-uniform matroids.

\begin{proof}[Proof of  Corollary \ref{exc}]

Lemma \ref{allem} implies that if every element of $M$ is a loop or a coloop, then it is an excluded minor of the class of almost $(k,l)$-uniform matroids if and only if $M\cong U_{k+1,k+1}\oplus U_{0,l}$ or $M\cong U_{k,k}\oplus U_{0,l+1}$. Now, suppose that $M$ is an excluded minor of the class of almost $(k,l)$-uniform matroids with an element $v$ that is neither a loop nor a coloop. Since $M-v$ and $M/v$ are both almost $(k,l)$-uniform, the deletion-contraction recursion $T_M(x,y)=T_{M-v}(x,y)+T_{M/v}(x,y)$ implies that $t_{k,l}^M\le 2$ and $t^M_{i,j}=0$ if $(i,j)>(k,l)$ by Theorem \ref{almuni}. Since $M$ is not an almost $(k,l)$-uniform matroid, $t_{k,l}^M=2$ and $t^{M-v}_{k,l}=1, t^{M/v}_{k,l}=1$. The converse direction is straightforward.
    
\end{proof}

Next result proposes a relation between excluded minors of almost $(k,l)$-uniform matroids and excluded minors of almost $(k+1,l)$-uniform matroids.

\begin{corollary}\label{cor:1}

Following two statements are equivalent.

\begin{enumerate}
    \item $N$ is an excluded minor of almost $(k,l)$-uniform matroids.
    \item $N\oplus U_{1,1}$ is an excluded minor of almost $(k+1,l)$-uniform matroids.
\end{enumerate}

\begin{proof}

It follows from the deletion-contraction recursion of the Tutte polynomial and Corollary \ref{exc}.
    
\end{proof}
    
\end{corollary}

The dual statement of Corollary \ref{cor:1} implies that $N$ is an excluded minor of almost $(k,l)$-uniform matroids if and only if $N\oplus U_{0,1}$ is an excluded minor of almost $(k,l+1)$-uniform matroids.

\section{A construction of some excluded minors of almost $(k,l)$-uniform matroids}\label{sec:cons}

As stated in the introduction, focusing on loopless and coloopless excluded minors of almost $(k,l)$-uniform matroids suffices, as demonstrated by Corollary \ref{cor:1}. In this section, we find some loopless and coloopless excluded minors of the class of almost $(k,l)$-uniform matroids extending previous results on excluded minors of almost uniform matroids and almost paving matroids by \citet{vega2020excluded}.

Let $\tau(M)$ denote the truncation of the matroid $M$ as it appeared in Theorem \ref{exc1}. Now, we construct infinitely many excluded minors of almost $(k,l)$-uniform matroids by proving Theorem \ref{exc1}. Following two preliminary facts about uniformity of truncation are required beforehand.

\begin{lemma}\label{lem:trun}

If $M$ is $(k,l)$-uniform, then $\tau(M)$ is $(k,l)$-uniform.

\begin{proof}

If $M$ is $(k,l)$-uniform, then it is $(k+1,l)$ uniform. Hence, it is enough to prove that if a matroid $M$ is $(k,l)$-uniform for some $k>1$, then its truncation $\tau(M)$ is $(k-1,l)$-uniform. Suppose that $M$ is a $(k,l)$-uniform matroid on $[n]$. By the second condition of $(k,l)$-uniform matroids, this is equivalent to that every corank $k$ flat of $M$ has nullity less than $l$. On the other hand, it is clear that a subset of $[n]$ is a flat of $\tau(M)$ if and only if it is a flat of $M$ that is not a hyperplane. Thus, corank $k$ flats of $M$ are exactly corank $k-1$ flats of $\tau(M)$. Combining these two facts, every corank $k-1$ flat of $\tau(M)$ has nullity less than $l$. This implies that $\tau(M)$ is $(k-1,l)$-uniform.
     
\end{proof}
    
\end{lemma}

\begin{lemma}\label{lem:mid}

For a positive integer $m$, let $M$ be a matroid of rank $d+m\ge m$ on the set $E$ and $U$ be a uniform matroid of rank $e+m\ge m$ on the set $E'$. Fix an order such that every element of $E$ is smaller than any element of $E'$. If a basis $B$ of the matroid $\tau^m(N\oplus U)$ satisfies $|B\cap E|\neq d$, then every element, that is externally active with respect to $B$, is contained in the set $E$. And if $|B\cap E|\neq d+m$, then every element, that is internally active with respect to $B$, is contained in the set $E$.

\begin{proof}

Bases of the matroid $\tau^m(M\oplus U)$ are unions $S\cup S'$ of two independent sets $S, S'$ of each matroids $M, U$ satisfying $|S|+|S'|=d+e+m$. If $|S|\neq d+m$, then there is an element $w\in E\setminus S$ such that $S\cup\{w\}$ is an independent set of $N$. This implies that an element $v$ of $S'$ cannot be internally active with respect to $S\cup S'$ because $S\cup S'\cup\{w\}\setminus \{v\}$ is a basis of $\tau^m(N\oplus U_{k+m,k+l+m})$ which implies that the fundamental cocircuit $C^*(v,S\cup S')$ contains $w$ which is smaller than $v$. If $|S|\neq d$, then for every element $w\in E'\setminus S'$, $S\cup S'\cup \{w\}$ is a circuit of the matroid $\tau^m(M\oplus U)$. This implies that fundamental circuit $C(w,S\cup S')$ contains $S$ which implies that $w$ is not externally active with respect to the basis $S\cup S'$. 
    
\end{proof}
    
\end{lemma}

\begin{proof}[Proof of Theorem \ref{exc1}]

First, assume that $m=0$. Let $B\cup B'$ be a basis of $N\oplus U_{k,k+l}$. The following equalities hold by the definition of direct sum \begin{align*}
    \inte_{N\oplus U_{k,k+l}}(B\cup B') &= \inte_N(B)+\inte_{U_{k,k+l}}(B') \\
    \exte_{N\oplus U_{k,k+l}}(B\cup B') &= \exte_N(B)+\exte_{U_{k,k+l}}(B').
\end{align*}

If $B$ is neither the smallest basis nor the largest basis of $N$, then $\inte_N(B)<k$ and $\exte_N(B)<l$. Note that any basis of a uniform matroid is either internal activity zero or external activity zero. Therefore, either $\inte_{N\oplus U_{k,k+l}}(B\cup B')<k$ or $\exte_{N\oplus U_{k,k+l}}(B\cup B')<l$.

Now, suppose that $B$ is the smallest basis of $N$. Then $\inte_N(B)=k$ and $\exte_N(B)=0$. Then  $B\cup B'$ is a $(k,l)$-basis if and only if $B'$ is the largest basis of $U_{k,k+l}$. And if $B$ is the largest basis of $N$, then $B\cup B'$ is a $(k,l)$-basis if and only if $B'$ is the smallest basis of $U_{k,k+l}$. Therefore $N\oplus U_{k,k+l}$ satisfies the second condition of Corollary \ref{exc}.

Deleting an element of $U_{k,k+l}$, the resulting matroid becomes $N\oplus U_{k-1,k+l-1}$. No basis of $U_{k-1,k+l-1}$ has internal activity $k$ and therefore a basis of $N\oplus U_{k-1,k+l-1}$ is $(k,l)$-basis if and only if it is the union of the largest basis of $U_{k-1,k+l-1}$ and the smallest basis of $N$. For an element $v$ of $N$, the deletion $N-\{v\}$ is a matroid of rank $k$ because $v$ is neither a loop or a coloop. Hence, no basis of $N-\{v\}$ has $l$ external activity. Thus, the only $(k,l)$-basis of the matroid $(N-\{v\})\oplus U_{k,k+l}$ is the union of the smallest basis of $U_{k,k+l}$ and the largest basis of $N-\{v\}$. This proves the third condition of Corollary \ref{exc}. Therefore $U_{k,k+l}\oplus N$ is an excluded minor of almost $(k,l)$-uniform matroids.

For a positive integer $m$, let $N$ be loopless and coloopless matroid of rank $k+m$ on the set of size $k+l+m$. Let $E, E'$ be the underlying set of $N, U_{k+m,k+l+m}$ respectively. Fix an order such that every element of $E$ is smaller than any element of $E'$. Hence, we may assume that $E=[k+l+m]$ and $E'=[2k+2l+2m]\setminus [k+l+m]=\{k+l+m+1,\ldots,2k+2l+2m\}$.

First consider the case $|S|\neq k$ and $|S|\neq k+m$. By Lemma \ref{lem:mid}, every internally active or externally active element is contained in the set $E$. For an element $v\in E\setminus S$, if $S\cup \{v\}$ is an independent set of $M$, then $\{v\}\cup S\cup S'$ is a circuit. Therefore, $v$ is externally active if and only if $v$ is the smallest element of the circuit $\{v\}\cup S\cup S'$. By our choice of ordering, this is equivalent to $v$ being the smallest element of the set $\{v\}\cup S$. If $S\cup \{v\}$ is not independnent, then a circuit $C$ of $N$ is contained in the set $S\cup\{v\}$. This circuit $C$ is the fundamental circuit of $v$ with respect to the basis $S\cup S'$ in the matroid $\tau^m(N\oplus U_{k+m,k+l+m})$. Thus, $v$ is externally active if and only if $v$ is the smallest element of $C$. In conclusion, $v$ is externally active with respect to the basis $S\cup S'$ if and only if $v$ is externally active with respect to the basis $S$ of the truncation $\tau^{k+m-|S|}(N)$. Dually, $v\in S$ is internally active with respect to the basis $S\cup S'$ in the matroid $\tau^m(N\oplus U_{k+m,k+l+m})$ if and only if $v$ is internally active with respect to the basis $S$ of the truncated matroid $\tau^{k+m-|S|}(N)$. Thus, the internal activity(resp. external activity) of $S\cup S'$ is the the internal activity(resp. external activity) of $S$ in $\tau^{k+m-|S|}(N)$. By Corollary \ref{cor:max}, $N$ is $(k,l)$-uniform and this implies that $\tau^{k+m-|S|}(N)$ is also $(k,l)$-uniform by Lemma \ref{lem:trun}. Hence, the inequality $(\inte(S\cup S'),\exte(S\cup S'))<(k,l)$ holds.

Now, suppose that $|S|=k$. Lemma \ref{lem:mid} implies that every element of $S'$ is not internally active with respect to $S\cup S'$. Assume that $S\cup S'$ is a basis of type $(i,j)\ge (k,l)$. Then every element of $S$ must be internally active with respect to $S\cup S'$. Thus the only possibility in this case is $S$ being the first basis of $\tau^m(N)$. Therefore, every element of $E-S$ can not be externally active with respect to $S\cup S'$. Consequently, this compels every element within $E'-S'$ to possess external activity with respect to $S\cup S'$. This implies that $S'$ must be the largest basis of $U_{k+m,k+l+m}$. Hence, there is only one basis of type $(k,l)$ with $|S|=k$. Similarly, there is also only one basis of type $(k,l)$ for the case $|S|=k+m$. Therefore, the matroid $N$ satisfies $t^N_{k,l}=2$ and $t^N_{i,j}=0$ if $(i,j)> (k,l)$. This establishes the second condition of Corollary \ref{exc}.

To finish the proof, it remains to prove that for any element $v\in E\cup E'$, $t^{\tau^m(N\oplus U_{k+m,k+l+m})-\{v\}}_{k,l}$ is $1$. For every element $v$, the matroid $\tau^m(N\oplus U_{k+m,k+l+m})-\{v\}$ obtained by deleting an element $v$ is either $\tau^m((N-\{v\})\oplus U_{k+m,k+l+m})$ or $\tau^m(N\oplus U_{k+m,k+l+m-1})$. Since $v$ is not a coloop of $N$, $N-\{v\}$ is a matroid of rank $k+m$ on the set of size $k+l+m-1$. Again, Lemma \ref{lem:mid} implies that if a basis $S\cup S'$ of $\tau^m((N-\{v\})\oplus U_{k+m,k+l+m})$ satisfies $|S|\neq k,k+m$, then every internally active or externally active element is contained in the set $E-\{v\}$. And $v\in E$ is internally active(resp. externally active) with respect to $S\cup S'$ iff $v$ is internally active(resp. externally active) with respect to the basis $S$ of the matroid $\tau^{k+m-|S|}(N-\{v\})$. It is clear that the matroid $N-\{v\}$ is $(k,l)$-uniform and so is its truncation by Lemma \ref{lem:trun}. This leads to the inference that $S\cup S'$ cannot be a $(i,j)$-basis for $(i,j)\ge (k,l)$. And if $|S|=k+m$, then Lemma \ref{lem:mid} implies that every externally active element of $S\cup S'$ is contained in $E$ which implies $\exte(S\cup S')\le |E|-|S|=l-1$. Hence, the only $(k,l)$-basis of $\tau^m((N-\{v\})\oplus U_{k+m,k+l+m})$ is the union of the smallest basis of $\tau^m(N-\{v\})$ and the largest basis of $U_{k+m,k+l+m}$. A similar proof is applicable to the case $\tau^m(N\oplus U_{k+m,k+l+m-1})$.
    
\end{proof}

The most simple loopless and coloopless excluded minor of almost $(1,1)$-uniform, or just almost uniform, matroids can be constructed using Theorem \ref{exc1} which was previously missing from \cite[Theorem 3.13]{lara2020exploring}.

\begin{example}

$U_{1,2}\oplus U_{1,2}$ is an excluded minor of almost uniform matroids. For every element $v$, deleting $v$ from this matroid results in $U_{1,2}\oplus U_{1,1}$. Deleting the coloop of $U_{1,2}\oplus U_{1,1}$ leaves us with $U_{1,2}$, which is a uniform matroid. For an element that is not a coloop, deleting the element results in $U_{2,2}$ which is also uniform. Therefore $U_{1,2}\oplus U_{1,1}$ is almost uniform. Contraction by any element $v$ results in $U_{1,2}\oplus U_{0,1}$ and similar argument shows that this matroid is also almost uniform.
    
\end{example}

In general, given a loopless and coloopless matroid $N,N'$ of rank $k+m$ on the set $[k+l+m]$, the matroid $\tau^m(N\oplus N')$ may fail to be an excluded minor of almost $(k,l)$-uniform matroids.

\begin{example}

Let $N=N'=U_{1,2}\oplus U_{1,2}$. Then $N\oplus N'=U_{1,2}\oplus U_{1,2}\oplus U_{1,2}\oplus U_{1,2}$ is not an excluded minor of almost $(2,2)$-uniform matroids since its Tutte polynomial is $(x+y)^4=x^4+4x^3y+6x^2y^2+4xy^3+y^4$.
    
\end{example}

\section{Further questions}\label{sec:q}

In Section \ref{sec:cons}, we constructed infinitely many excluded minors of almost $(k,l)$-uniform matroids. Is there any other excluded minor of almost $(k,l)$-uniform matroids exists? 

Let the class $A^h\mathcal{U}(k,l)$, encompassing matroids that are almost $(k,l)$-uniform with the term `almost' iterated $h$ times. Is the class  $A^h\mathcal{U}(k,l)$ also characterized by Tutte polynomial coefficients? Identify excluded minors of the class $A^h\mathcal{U}(k,l)$.

\citet[Conjecture~1.6]{MR2821559} proposed a conjecture that predicts the proportion of paving matroids among all matroids on the set $[n]$ will approach $1$ as $n$ increases to infinity, see \cite{MR3367125,MR3367294} for some related results. Motivated by this conjecture, we propose a new conjecture which suggests that almost all $(k+1,l)$-uniform matroids are $(k,l)$-uniform or dually, almost all $(k,l+1)$-uniform matroids are $(k,l)$-uniform unless $(k,l)=(1,1)$. For the special case of $(k,l)=(2,1)$, this statement is weaker than the conjecture of \citet[Conjecture~1.6]{MR2821559}.

\begin{conjecture}

Let $m_n(k,l)$ be the number of all $(k,l)$-uniform matroids on the set $[n]$. If $(k,l)\neq (1,1)$, then the following limit holds:

\begin{itemize}
    \item $\lim_{n\to\infty} \frac{m_n(k,l)}{m_n(k+1,l)}\to 1$; or dually,
    \item $\lim_{n\to\infty} \frac{m_n(k,l)}{m_n(k,l+1)}\to 1$.
\end{itemize}

\end{conjecture}

\bibliographystyle{plainnat}
\bibliography{ref}
\end{document}